\documentclass[12pt]{article}

\usepackage[margin=0.9in]{geometry}
\usepackage{pgf,tikz}
\usepackage{mathrsfs}
\usepackage{amsmath}
\usepackage{amssymb}
\usepackage{hyperref}
\usepackage[utf8]{inputenc}
\usepackage[english]{babel}
\usepackage{float}
\usepackage{epstopdf}
\usetikzlibrary{arrows}

\usepackage{amsthm}
\newtheorem{definition}{Definition}[section]
\newtheorem{proposition}{Proposition}[section]
\newtheorem{theorem}{Theorem}[section]
\newtheorem{corollary}{Corollary}[theorem]
\newtheorem{lemma}[theorem]{Lemma}

\newenvironment{remark}[1][Remark]{\begin{trivlist}
\item[\hskip \labelsep {\bfseries #1}]}{\end{trivlist}}
\usepackage{graphicx, url}

\usepackage{graphicx}

\usepackage{enumitem}

\begin{document}
\title{Triangles with Vertices Equidistant to a Pedal Triangle}
\author{Xuming Liang, Ivan Zelich}
\date{}
\maketitle

\begin{abstract}
	In this paper, we present a synthetic solution to a geometric open problem involving the radical axis of two strangely defined circumcircles. The solution encapsulates two generalizations, one of which uses a powerful projective result relating isogonal conjugation and polarity with respect to circumconics.
\end{abstract}

\section{Introduction}

The 2012 volume of the Journal of Classical Geometry featured an open problem posed by Lev Emelyanov \cite{openproblem}, which with some change in notation can be stated as follows

\begin{theorem}
	Let $N$ be the Nagel point of a triangle ABC. Let $A_0, B_0$ and $C_0$ be the touch points of excircles with the sides $BC, CA, AB$ respectively. Let us consider six points: $A_1$ and $A_2$ on the $BC$, $B_1$ and $B_2$ on the $CA$, $C_1$ and $C_2$ on the $AB$, such that 
	\[
	A_1A_0 = A_2A_0 = B_1B_0 = B_2B_0 = C_1C_0 = C_2C_0.
	\]
	Points $A_1, B_1, C_1$ lie on the rays $A_0C, B_0A, C_0B$ and $A_2, C_2, B_2$ lie on the rays $A_0B, C_0A, B_0C$ respectively.
	Then the radical axis of circumcircles $\odot (A_1B_1C_1) and \odot(A_2B_2C_2)$ passes through the Nagel point $N$, centroid and incenter of $ABC$.
\end{theorem}

\begin{figure}[H]
	\centering
	\includegraphics[scale = 0.6]{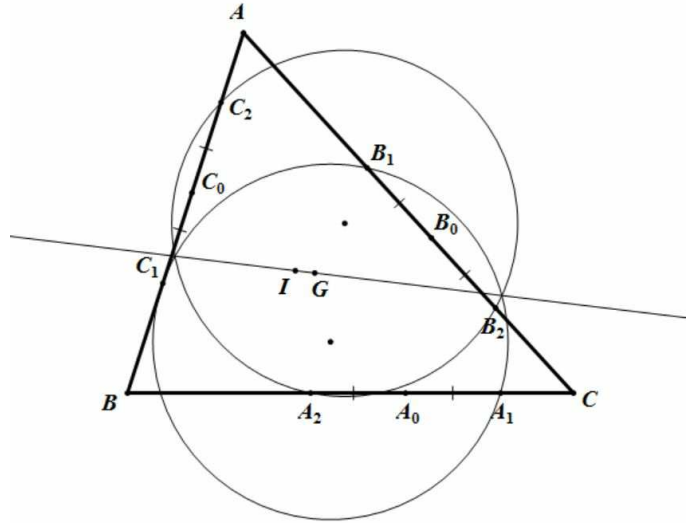}
	\caption{Theorem 1.1}
	\label{T11}
\end{figure}

In this paper, we will present two generalizations to this proposition. It turns out if $A_0B_0C_0$ is the pedal triangle of an arbitrary point $P$ on the plane of $ABC$, then the radical axis in question would always pass through a fixed point (Theorem 2.1). Furthermore, if $P$ lies on the line connecting the incenter and circumcenter of $ABC$, then the radical axis becomes fixed(Theorem 3.1). Note that $A_0B_0C_0$ in the problem statement above is actually the pedal triangle of the Bevan point of $ABC$, which lies on the incenter-circumcenter line, so the generalization partially explains why the radical axis does not change. It will take additional insights from a powerful projective result (Proposition 3.3) to reveal that the radical axis coincides with the Nagel line\textemdash the line that contains the Nagel point, centroid, and incenter.

\section{General Case: $P$ is arbitrary}

The purpose of this section is to prove one the main results captured by the paper, stated below.

\begin{theorem}
	Let $P$ be an arbitrary point on the plane of triangle $ABC$. Let $A_PB_PC_P$ be the pedal triangle of $P$ on $ABC$. Define $A_1,A_2,B_1,B_2,C_1,C_2$ such that for some $x > 0$ 
	\[ 
	x = A_PA_1 = A_PA_2 = B_PB_1 = B_PB_2 = C_PC_1 = C_PC_2,
	\]
	with $A_1,B_1,C_1$ lying on rays $A_PC, B_PA, C_PB$, and $A_2,B_2,C_2$ lying on rays $A_PB, B_PC, C_PA$. As $x$ varies, the radical axis of $\odot A_1B_1C_1$ and $\odot A_2B_2C_2$ passes through a fixed point.
\end{theorem}
The proof will reveal the geometric nature of the triangles $A_1B_1C_1$ and $A_2B_2C_2$ that makes this theorem true. The approach taken here involves identifying the fixed point and proving that the it exerts equal power with respect to the triangles' circumcircles. This was done by applying a metric lemma that relates circumradius of a reflection triangle and distances to an orthocenter. The metric conditions in the theorem translates straight-forwardly to the application of the lemma with minimal algebraic manipulations and calculations, granting the theorem a truely synthetic proof. 

\subsection{Preliminary Definitions and Propositions}
\begin{definition}
	Two triangles $ABC$ and $DEF$ are said to be \textbf{orthologic} if the perpendiculars from $A,B,C$ to sides $EF, DF, DE$ are concurrent. It can be shown that this relation is symmetric, namely the perpendiculars from $D,E,F$ to sides $BC,AC,AB$ are also concurrent. The first point of concurrency is called the orthologic center of $ABC$ w.r.t $DEF$ and is denoted $O_{ABC}(DEF)$.
\end{definition}

\begin{definition}
	Given a triangle $ABC$ and a point $P$ on its circumcircle. The reflection of $P$ over the sides of $ABC$ lie on a line called \textbf{Steiner's line} of $P$ w.r.t $ABC$. In other words, the Steiner's line is the Simson's line of $P$ dilated by a factor of 2. It is well known that the Steiner's line contains the orthocenter of $ABC$, and is perpendicular to isogonals of $P$ w.r.t $\angle A,\angle B , \angle C$.
\end{definition}

We now introduce an important property on a local configuration in the problem. 

\begin{proposition}
	Given a triangle $ABC$, suppose $B',C'$ are two points on $AB,AC$ respectively such that $BB' = CC'$ and they lie on opposite sides of $BC$. If $D$ is the intersection of the angle bisector of $\angle A$ with $\odot (ABC)$, prove that $DB' = DC'$
\end{proposition}

\begin{proof}
Since $AD$ bisects $\angle BAC$, we have $DB = DC$ and by construction of $B',C'$, $\angle B'BD = \angle C'CD$. Therefore, $BB' = CC'$ implies that $\triangle DBB'\cong \triangle DCC'$ and thus $DB' = DC'$. 
\end{proof}

The following will be the central property used to compute the power of the fixed point with respect to the circumcircles of $A_1B_1C_1$ and $A_2B_2C_2$. Its proof is slightly complicated and will be omitted here. 
\begin{proposition}\cite{orthocenterlemma}
	Let $H,O$ be the orthocenter and circumcenter of $\triangle ABC$. Suppose $P,Q$ are isogonal conjugates w.r.t $\triangle ABC$. Prove that 
	\[
	R_{PQ}^2 = HQ^2 + HP^2 + |R^2-OH^2|
	\]
	where $R, R_{PQ}$ are the circumradii of $\triangle ABC$ and the reflection triangle of $P$(or $Q$) w.r.t $\triangle ABC$ respectively. Note that $|R^2-OH^2|$ is simply the power of $H$ w.r.t $\odot (ABC)$.
\end{proposition}

\subsection{Proof}

Let $I$ be the incenter of $ABC$. From $P$ drop pedals onto $AI,BI,CI$ and denote them by $A',B',C'$ respectively. With $P,A'\in \odot(AB_PC_P)$, proposition 2.1 states that $A'B_1 = A'C_1$ and $A'B_2 = A'C_2$. Symmetrically, we can obtain $B'A_i = B'C_i, C'A_i = C'B_i$ for $i = 1,2$. To rid the suspense, we claim that the fixed point that lies on the radical axis of $\odot(A_1B_1C_1), \odot(A_2B_2C_2)$ is the orthocenter of $A'B'C'$, which we will denote by $H'$. This first lemma opens the door for $A_1B_1C_1$ to relate with other triangles.

\begin{lemma}
	Triangles $A'B'C'$ and $A_1B_1C_1$ are orthologic. In addition, the two orthologic centers between the triangles are isogonal conjugates wrt $A'B'C'$.	
\end{lemma}

\begin{proof}
Let $O_1$ denote the circumcenter of $A_1B_1C_1$. We claim that $O_1 = O_{A'B'C'}(A_1B_1C_1)$. Indeed, $A'$ lies on the perpendicular bisector of $B_1C_1$, which $O_1$ also lies on, so $A'O_1\perp B_1C_1$. Symmetrically, we can obtain $B'O_1\perp A_1C_1, C'O_1\perp A_1B_1\implies$ $A'B'C'$ and $A_1B_1C_1$ are orthologic. 

Now Let $Q_1 = O_{A_1B_1C_1}(A'B'C')$. Since $B_1Q_1\perp A'C', C_1Q_1\perp A'B'$, thus 
\begin{align*}
\angle (Q_1B_1, Q_1C_1) &=  \angle (A'C',A'B') \\
&=  \angle (IC',IB') \tag{$A',B',C',P,I$ are concyclic}\\
&=  90-\frac{\angle A}{2}\\
&= \frac{180-\angle A}{2}\\
&= \frac{\angle (A'B_1, A'C_1)}{2}.
\end{align*}
It follows that $A'$ is the circumcenter of $B_1C_1Q_1$. Thus, $B_1,C_1$ are the reflection of $Q_1$ over $A'C',A'B'$ respectively. It follows by symmetry that $A_1$ is the reflection of $Q_1$ over $B'C'$. These result can be summerized as follows: $A_1B_1C_1$ is the reflection triangle of $Q_1$ w.r.t $A'B'C'$. By a well-known fact of isogonal conjugates, $O_1$, the circumcenter of the reflection triangle of $Q_1$, is isogonal conjugate with $Q_1$ w.r.t $A'B'C'$; this proves second part of the proposition. 
\end{proof}

\begin{figure}[H]
	\centering
	\includegraphics[scale = 0.7]{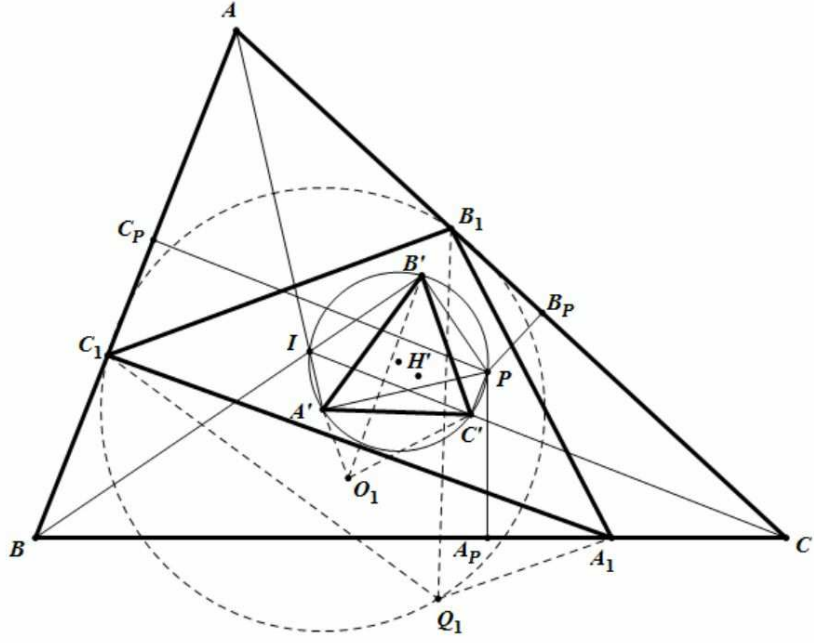}
	\caption{Proof of Lemma 2.2. Initial set up for the solution.}
	\label{L22Proof}
\end{figure}

\begin{lemma}
	If $Q_1 = O_{A_1B_1C_1}(A'B'C')$ and $Q_2 = O_{A_2B_2C_2}(A'B'C')$, then $H'Q_1 = H'Q_2$.	
\end{lemma}

\begin{proof}
By taking $x = 0$, $A_1B_1C_1, A_2B_2C_2$ coincide with $A_PB_PC_P$, so in light of Proposition Lemma 2.2, we define $Q = O_{A_PB_PC_P}(A'B'C')$. Let $A'H'\cap \odot (A'B'C') = H_A'\ne A'$, thus $H', H_A'$ are symmetric about $B'C'$. We claim that $H_A'\in PA_P$, which follows from angle chasing $\angle (PA', PA_P) = \angle (PA', PH_A')$. In the proof above, we showed that $A_1,Q_1$($A_2,Q_2$) are symmetric about $B'C'$. Therefore, a reflection over $B'C'$ produces $H_A'A_1A_2A_P \mapsto H'Q_1Q_2Q$, and the lemma follows.
\end{proof}

\begin{figure}[h]
	\centering
	\includegraphics[scale = 0.7]{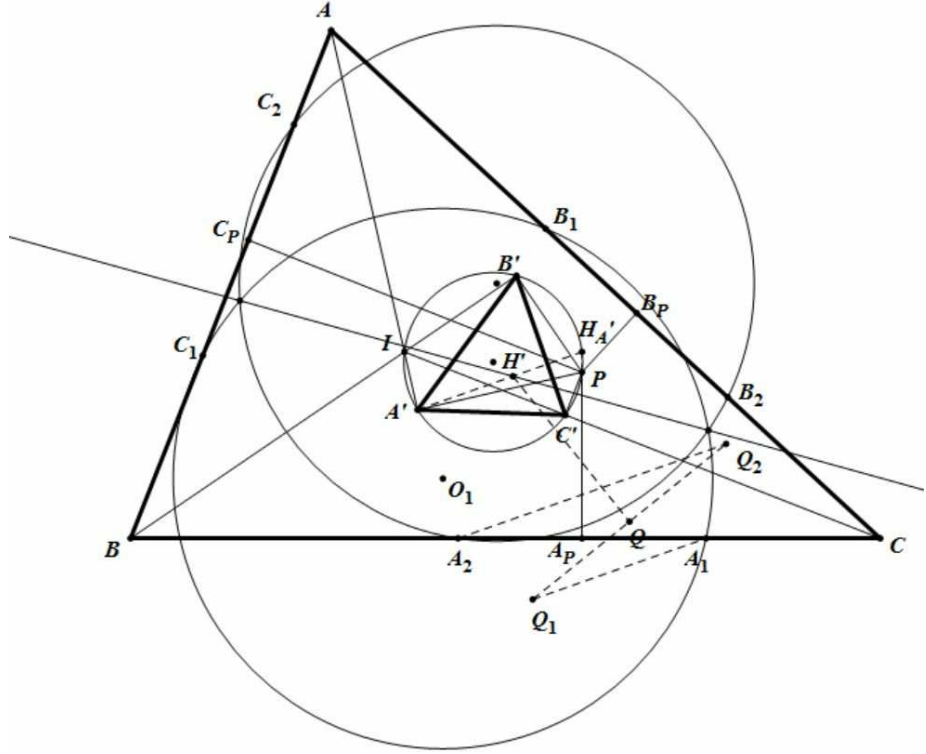}
	\caption{Proof of Lemma 2.3}
	\label{L23Proof}
\end{figure}

\begin{remark}
	Notice that we have proven that $H'Q$ is the Steiner's line of $P$ with respect to $A'B'C'$. In addition, as $A_1,A_2$ varies with $x$, $Q_1,Q_2$ varies on a fixed line through $Q$ and perpendicular to $H'Q$. This observation will be used in the next section.
\end{remark}

We are now ready to prove the main result.

\begin{proof}
We proved in Lemma 2.2 that $Q_1,O_1$ are isogonal conjugates w.r.t $A'B'C'$, and $A_1B_1C_1$ is the reflection triangle of $Q_1$ w.r.t $A'B'C'$. Applying Proposition 2.2, we have
\[
O_1A_1^2 = O_1H'^2+Q_1H'^2+p
\]
where $p$ is the power of $H'$ w.r.t $\odot (A'B'C')$.
Symmetrically, we can obtain
\[
O_2A_2^2 = O_2H'^2 + Q_2H'^2 + p
\]
Since Proposition asserts that $H'Q_1 = H'Q_2$, it follows by subtracting the two equations that
\[
O_1A_1^2 - O_1H'^2 = O_2A_2^2 - O_2H'^2.
\]
However, this equality is equivalent to $H'$ having equal power wrt $\odot (A_1B_1C_1)$ and $\odot (A_2B_2C_2)$. This implies $H'$ lies on the radical axis of the two circles and the result is proven.
\end{proof}

\section{Special Case: When $P$ lies on $IO$}

This section aims to investigate the case where $P$ lies on $IO$, where $I,O$ are incenter and circumcenter of $ABC$. The main result is stated as follows.  

\begin{theorem}
	Let $P$ be a point that lies on $IO$, where $I$ and $O$ are the incenter and circumcenter of $ABC$ respectively. Let $A_PB_PC_P$ be the pedal triangle of $P$ on $ABC$. Define $A_1,A_2,B_1,B_2,C_1,C_2$ such that for some $x > 0$ 
	\[ 
	x = A_PA_1 = A_PA_2 = B_PB_1 = B_PB_2 = C_PC_1 = C_PC_2,
	\]
	with $A_1,B_1,C_1$ lying on rays $A_PC, B_PA, C_PB$, and $A_2,B_2,C_2$ lying on rays $A_PB, B_PC, C_PA$. As $x$ varies, the radical axis of $\odot A_1B_1C_1$ and $\odot A_2B_2C_2$ is a fixed line.
\end{theorem}
Proving this would bring us one step closer to deducing our open problem. Surprisingly, efforts to establish this theorem lead to explorations into fruitful projective results that strayed relative far from the theorem at hand. 

We now pick up where we left off from the previous section to analyze the current theorem. Since we proved that the radical axis contains a fixed point, the natural next step is to show that the radical axis is always parallel to a fixed line. This is equivalent to showing that the line connecting circumcenters of $\triangle A_1B_1C_1$ and $\triangle A_2B_2C_2$, $O_1$ and $O_2$ respectively, is always parallel to a fixed line as $x$ varies. Keeping the same notations as the previous section, recall from Lemma 2.2 that $O_1,O_2$ are isogonal conjugates of $Q_1,Q_2$, a pair of points on a line through $Q$ (defined as $O_{A_PB_PC_P}(A'B'C')$) and perpendicular to the Steiner's line of $P$ w.r.t $\triangle A'B'C'$. Additionally, $Q_1,Q_2$ are equidistant from $Q$. Therefore, it suffices to find out what is special about the point $Q$ and Steiner's line of $P$ when $P$ lies on $IO$. The next proposition helps answer this.

\begin{proposition}
	Suppose $P$ is a point on $IO$. Define $A',B',C'$ to be the pedals from $P$ onto $AI,BI,CI$ respectively. Then the Steiner's line of $P$ w.r.t. $\triangle A'B'C'$ is the Euler line of $\triangle A'B'C'$.
\end{proposition}

Notice that as $P$ moves along $IO$, the corresponding configurations for $A'B'C'\cup P$ are homothetic through $I$. Therefore, it suffices to prove that this proposition is true for one point on $IO$. For the sake of simplicity, we will choose $P\equiv B_e$ the Bevan point of $\triangle ABC$, i.e. if $I_AI_BI_C$ is the excenter triangle of $ABC$, then $B_e,O,I$ are its circumcenter, nine-point center, and orthocenter respectively. Thus we can restate Proposition 3.1 in terms of the excenter triangle.

\begin{lemma}
	Given a triangle $ABC$ with circumcenter $O$, nine-point center $N$, and orthocenter $H$. Define $A^*,B^*,C^*$ to be the pedals from $O$ onto $AH,BH,CH$. Then the Steiner's line of $P$ w.r.t $A^*B^*C^*$ is the Euler line of $\triangle A^*B^*C^*$.  
\end{lemma}

\begin{figure}[H]
	\centering
	\includegraphics[scale = 0.7]{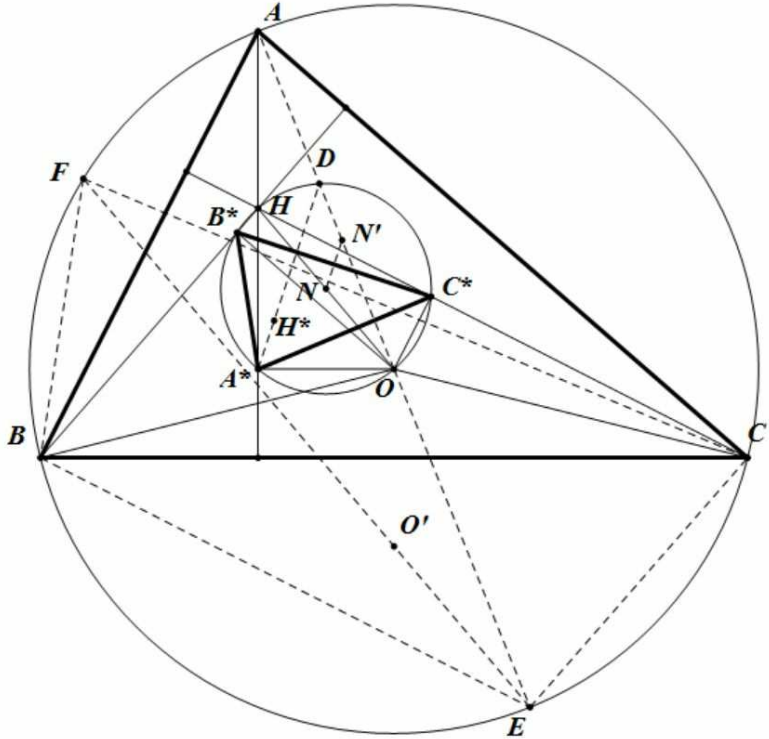}
	\caption{Proof of Lemma 3.2}
	\label{L32Proof}
\end{figure}

\begin{proof}
	Let $H^*$ be the orthocenter of $A^*B^*C^*$. Suppose $A^*H^*\cap \odot (A^*B^*C^*) = D\ne A^*$, then $D$ is the reflection of $H^*$ over $B^*C^*$. Clearly $N$ is the circumcenter of $A^*B^*C^*$, so let $N'$ be the reflection of $N$ over $B^*C^*$. By properties of the Steiner's line, it suffices to show that $O, D, N'$ are collinear.
	
	Angle chasing reveals that $A^*B^*C^*\sim ABC$, thus 
	\[
	\angle (OC^*, OD) = \angle (A^*C^*, A^*D) =\angle (AH, AC) = \angle (AB, AO) = \angle (OC^*, OA),
	\]
	which means $O,D,A$ are collinear. Hence it remains to show that $O,N',A$ are collinear.
	
	Let $O'$ be the reflection of $O$ over $BC$, and define $AO\cap \odot (ABC) = E\ne A, EO'\cap \odot(ABC) = F\ne E$. By properties of the nine-point center, $EO'\parallel HO$. Since $CE\parallel BH, BE\parallel CH$, it follows by angle chasing with cyclic quadrilaterals that $FBC\sim OC^*B^*$. Because $OO'$ corresponds to $NN'$ between the two similar triangles, we have $FBC\cup O'\sim OC^*B^*\cup N'$. Hence,
	\[
	\angle (OC^*, ON') = \angle (FB, FO') = \angle (AB, AO) = \angle (OC^*, OA),
	\]
	which implies that $O,N',A$ are collinear and proves the proposition.
\end{proof}

Now that Proposition 3.1 is proven, the implication is that the Steiner's line of $P$ w.r.t. $\triangle A'B'C'$ passes through the circumcenter of $\triangle A'B'C'$. It turns out that this is the property that makes Theorem 3.1 true, as exemplified by the following generalization we found.

\begin{proposition}
	Given a triangle $ABC$ and a line $l$ on the same plane. Let $O$ be the circumcenter of $ABC$, and let $T$ be the pedal from $O$ to $l$. Suppose $R$ and $S$ are two variable points on $l$ satisfying $TS = TR$. If $S', R'$ are the isogonal conjugates of $S, R$ w.r.t $ABC$ respectively, then $S'R'$ is parallel to a fixed line as $S,R$ vary on $l$.
\end{proposition}

This result would finish the proof for Theorem 3.1 when applying it to triangle $A'B'C'$ with $Q$ as $T$ and the line through $Q$ perpendicular to the Euler line (by Lemma 3.2) of $A'B'C'$ as $l$. We will devote the next subsection to proving Proposition 3.2 and give a formal proof of the theorem in the last subsection.

\subsection{Consequences of a Projective Generalization}

Under an isogonal conjugation with respect to a fixed triangle, a line is transformed into a circumconic of the triangle and vise versa \cite{isogonal1}. Thus we were encouraged to work on the projective plane containing conics. 

\begin{definition}
	With respect to a conic $\mathcal{C}$, two points $A,B$ are said to be conjugates if the polar of $A$ passes through $B$ or, equivalently, the polar of $B$ passing through $A$.
\end{definition}

We now present a powerful result relating isogonal conjugation and polarity.

\begin{proposition}
	Given a triangle $ABC$ and two lines $l_1, l_2$. Let $\psi$ denote the isogonal conjugation w.r.t $\triangle ABC$. Suppose $\psi(l_1) = \mathcal{C}_1, \psi(l_2) = \mathcal{C}_2$. If $X,Y$ are two points on $l_1$ and $\psi(X) = X', \psi(Y) = Y'$, then $X,Y$ are conjugates w.r.t $\mathcal{C}_2$ $\implies $ $X'Y'$ passes through the pole of $l_2$ w.r.t $\mathcal{C}_1$
\end{proposition}

\begin{proof}
	
	We first recall two properties of isogonal conjugation and conjugates. One, at least one of a conjugate pair is an external point, meaning two tangents can be drawn from the point to the respective conic. Two, isogonal conjugation preserves tangency, i.e. if $l$ is tangent to $\mathcal{C}$, then $\psi(\mathcal{C})$ is tangent to $\psi(l)$. 
	
	In light of the first property, WLOG assume $X$ is an external point w.r.t. $\mathcal{C}_2$; thus there exists two points $V,U\in \mathcal{C}_2$ such that $XV, XU$ are tangents. Take the following isogonal conjugations:
	\[
	\psi(V) = V', \psi(U) = U', \psi(VU) = \mathcal{C}_X, \psi(XV) = \mathcal{H}_V,  \psi(XU) = \mathcal{H}_U.
	\]
	Since $Y\in VU$, it follows that $Y' \in \mathcal{C}_X$. By the second property, $ \mathcal{H}_V,  \mathcal{H}_U$ are tangent to $l_2$ at $U', V'$ respectively. Consider the following cross-ratio equality(or projectivity if you prefer):
	\begin{align*}
	Y'(A,B;C,V') &=  U'(A,B;C,V') \tag{taken on $\mathcal{C}_X$} \\
	&= U'(A,B;C,U') \tag{taken on $\mathcal{H}_U$} \\
	&= X'(A,B;C,U') \tag{taken on $\mathcal{H}_U$} 
	\end{align*}
	This implies $X'U'\cap Y'V' = V^*\in \mathcal{C}_1$. Symmetrically, we can obtain $X'V'\cap Y'U' = U^*\in \mathcal{C}_1$. Hence $X'Y'$ intersect $V^*U^*$ at the pole of $l_2$ w.r.t $\mathcal{C}_1$.
\end{proof}

\begin{figure}[t]
	\centering
	\includegraphics[scale = 0.7]{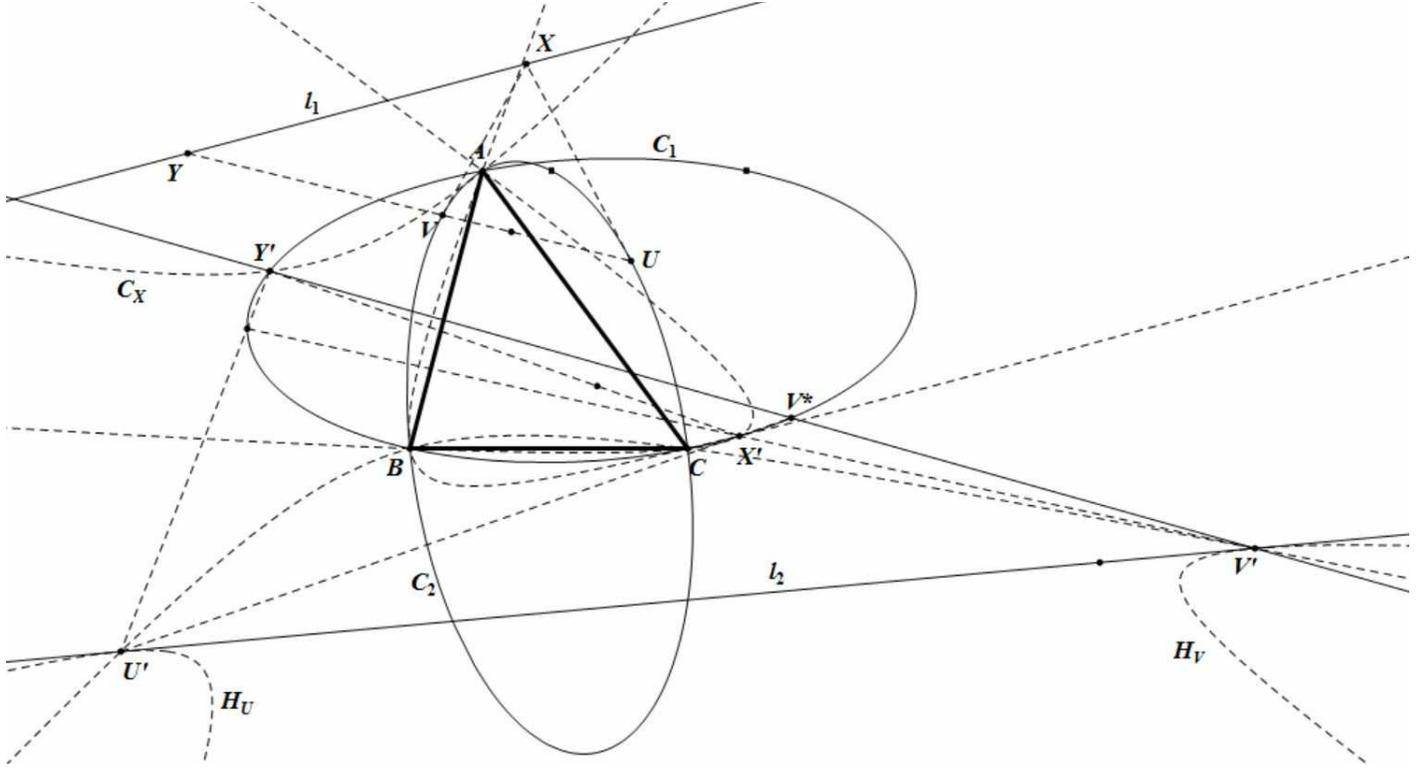}
	\caption{Proof of Proposition 3.3}
	\label{P33Proof}
\end{figure}

\begin{corollary}
	Let $X,Y$ be conjugates conjugates w.r.t $\odot (ABC)$. If $X', Y'$ are isogonal conjugates of $X,Y$ w.r.t $\triangle ABC$, then $X'Y'$ passes through the center of the circumconic defined by $A,B,C,X',Y'$.
\end{corollary}

\begin{proof}
	This is simply a consequence of Proposition 3.3 with $l_2$ as the line at infinity and $l_1$ as $XY$, since $\psi(l_2) = \odot(ABC)$ and the pole of $l_2$ w.r.t $\mathcal{C}_1$ is the center of $\mathcal{C}_1$.
\end{proof}

We can now prove Proposition 3.2.

\begin{figure}[H]
	\centering
	\includegraphics[scale = 0.7]{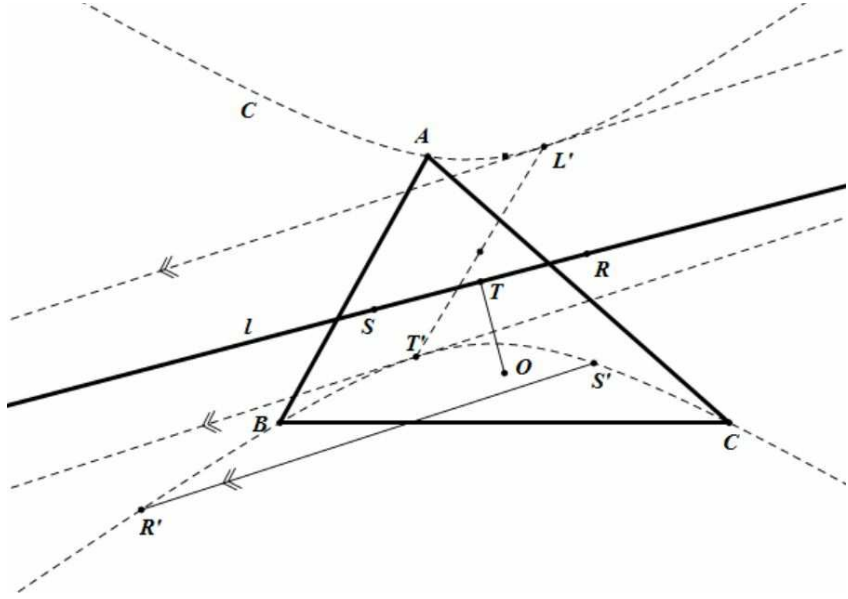}
	\caption{Proposition 3.2}
	\label{P32Proof}
\end{figure}

\begin{proof}[Proof of Proposition 3.2.]
	Let $l_\infty$ denote the point at infinity on $l$. Suppose $l$ is transformed to a circumconic $\mathcal{C}$ under isogonal conjugation w.r.t $ABC$. If $T',L'$ are the isogonal conjugates of $T, l_\infty$, then $T',L',X', Y'$ all lie on $\mathcal{C}$. Since $l_\infty$ and $T$ are conjugates w.r.t $\odot (ABC)$, Corollary 3.2.1 asserts that $T'L'$ passes through the center of $\mathcal{C}$, or equivalently, the pole of $T'L'$, denoted by $K$, lies on the line of infinity. 
	
	Observe that 
	\[
	A(T', L'; S', R') = A(T, l_\infty; S, R) = -1.
	\]
	This implies $S'R'$ passes through $K$. Hence $S'R'$ is parallel to the tangent at $l_\infty$ at $\mathcal{C}$, which is a fixed line. 
\end{proof}

\subsection{Proof}

We now give a formal proof of Theorem 3.1. We will heavily refer to results already established in Section 2.

\begin{proof}[Proof of Theorem 3.1]
	Define $A',B',C'$ to be the pedals from $P$ onto $AI,BI,CI$ respectively. Lemma 2.2 states that the following points exist:
	\[
	O_{A_1B_1C_1}(A'B'C') = Q_1, O_{A_2B_2C_2}(A'B'C') = Q_2, O_{A_PB_PC_P}(A'B'C') = Q.
	\]
	Furthermore, we proved in Lemma 2.3 that $Q_1,Q,Q_2$ all lie on a line, denoted by $q$, that is perpendicular to the Steiner's line of $P$ w.r.t $A'B'C'$ and the points satisfy $QQ_1 = QQ_2$. By Proposition 3.1, we know the circumcenter $O'$ of $A'B'C'$ lies on the Steiner's line. Thus, $Q$ can be viewed as the pedal from the $O'$ to $q$. From Lemma 2.2, we know the circumcenters $O_1, O_2$ of $\triangle A_1B_1C_1, \triangle A_2B_2C_2$ are isogonal conjugates of $Q_1, Q_2$ w.r.t $\triangle A'B'C'$. Applying Proposition 3.2 on triangle $A'B'C'$ with $q$ as $l$ and $Q$ as $T$, we obtain that $O_1O_2$ is parallel to a fixed line as $Q_1,Q_2$ varies $l$. Since the radical axis of $\odot (A_1B_1C_1)$ and $\odot (A_2B_2C_2)$ is perpendicular to $O_1O_2$ and passes through a fixed point, it follows that the radical axis is fixed.
\end{proof}

\begin{proposition}
Keeping the same notations as above, if $q$ is transformed to $\mathcal{C}$ under isogonal conjugation w.r.t $\triangle A'B'C'$, then we claim that $P, O_P\in \mathcal{C}$, where $O_P$ is the circumcenter of $A_PB_PC_P$. 
\end{proposition}

\begin{proof}
 On one hand, $O_P$ is the isogonal conjugate of $Q$, which lies on $q$; so $O_P\in C$. On the other hand, since $P\in \odot(A'B'C')$, the isogonal conjugate of $P$ is the point at infinity corresponding to the direction perpendicular to the Steiner line of $P$ w.r.t $A'B'C'$. Because $q$ is parallel to this direction, it contains the isogonal conjugate of $P$, which proves that $P\in \mathcal{C}$.	
\end{proof}

The following corollary follows directly from the proof of Proposition 3.2.

\begin{corollary}
Keeping the same notations as above, the fixed radical axis is perpendicular to the tangent to $O_P$ or $P$ on $\mathcal{C}$.
\end{corollary}

\section{Proof of the Open Problem}

The open problem focuses on when $P = B_e$, the Bevan point of $ABC$. It suffices to prove that the fixed line described in Theorem 3.1 is the Nagel line. We first prove that the fixed point characterized in section 2 lies on the Nagel line. In fact, the next proposition proves that it coincides with the Nagel point. Then, it remains to show that the Nagel line is perpendicular to the fixed line characterized in section 3.

\begin{proposition}
	Let $B_e$ be the Bevan point of $ABC$. Define $A',B',C'$ to be the pedals from $B_e$ onto $AI,BI,CI$ respectively. Then, the orthocenter of $A'B'C'$ coincides with the Nagel point of $ABC$.
\end{proposition}

Similar to Proposition 3.1, Proposition 4.1 will follow from a lemma written in the perspective of the excenter triangle of $ABC$.

\begin{lemma}
	Given a triangle $ABC$ with circumcenter $O$, orthocenter $H$, and orthic triangle $A_HB_HC_H$. Define $A^*,B^*,C^*$ to be the pedals from $O$ onto $AH,BH,CH$ respectively. If $AO$ intersects $B_HC_H$ at $J$, then the orthocenter $H^*$ of $A^*B^*C^*$ lies on $A_HJ$.
\end{lemma}

\begin{proof}
	Let $AO\cap \odot(A^*B^*C^*) = D\ne O$. Recall from the proof of Lemma 3.2 that $A^*D\perp B^*C^*$; thus $H^*\in A^*D$. Additionally, since $HO$ is the diameter of $ \odot(A^*B^*C^*)$, we have $HD\perp AO$, which implies $D\in \odot(AB_HC_HH)$. If $AH\cap \odot(ABC) = G\ne A$, then because $AD\perp B_HC_H$ and $ABC\sim AB_HC_H$, it follows that 
	\[
	\frac {AJ}{JD} = \frac {AA_H}{GA_H},
	\]
	implying $A_HJ\parallel GD$. 
	
	Note that it is a well known property of circumcenters that the distance of $O$ to $BC$ is half of $AH$, i.e. $2A^*A_H = AH$. Because $2GA_H = HG$, we get
	\[
	\frac {A^*A_H}{GA_H} = \frac {AH}{HG} = \frac {A^*H^*}{DH^*}
	\]  
	where the last equality follows from the similarity $ABC\cup H\cup G\sim A^*B^*C^*\cup H^*\cup D$. Notice that this implies $A_HH^*\parallel GD$. Hence, $A_HJ\parallel A_HH^*\implies H^*\in A_HJ$.
\end{proof}

\begin{figure}[t]
	\centering
	\includegraphics[scale = 0.7]{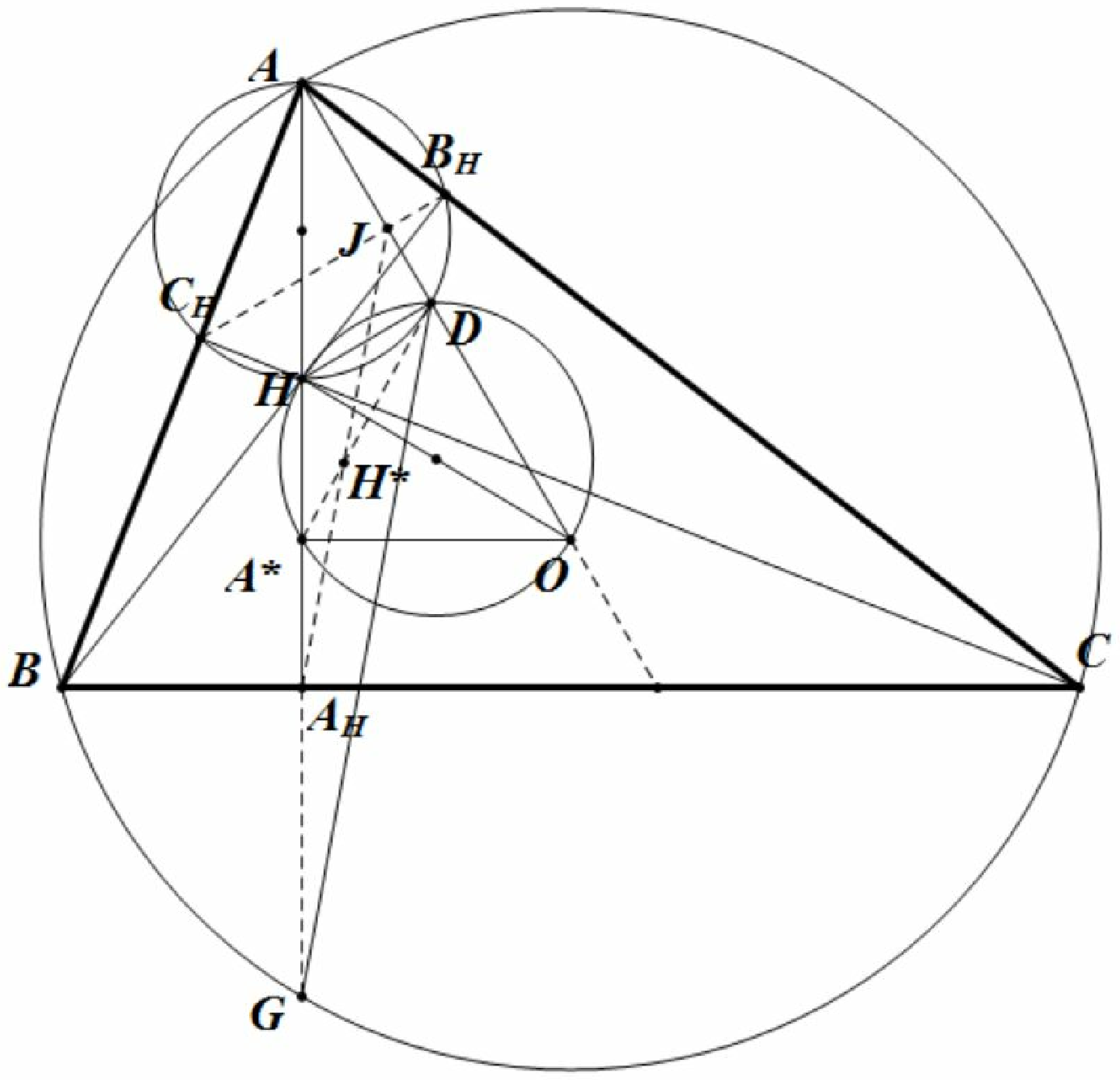}
	\caption{Proof of Lemma 4.1}
	\label{L41Proof}
\end{figure}

In the context of Proposition 4.1, note that if $A_0B_0C_0$ is the pedal triangle of $B_e$, then the Nagel point is where $AA_0, BB_0, CC_0$ concur. In Lemma 4.1, $O$ is the Bevan point of $A_HB_HC_H$, and $J$ is the pedal from $O$ to $B_HC_H$. Thus a symmetric application of the lemma implies that $H^*$ is the Nagel point of $A_HB_HC_H$, and Proposition 4.1 follows by a change in notation.

\begin{proposition}
	Let $O_0$ to be the circumcenter of $A_0B_0C_0$, and let $H'$ be the orthocenter of $A'B'C'$. Then $B_e, H', O_0$ are collinear (See Figure 8.)
\end{proposition}

\begin{proof}
	Let $B_e'$ denote the isogonal conjugate of $B_e$ w.r.t $ABC$. According to Proposition 4.1, $H'$ is the Nagel point of $ABC$, so by definition, $ABC$ and $A_0B_0C_0$ are perspective at $H'$. Observe that $B_e = O_{A_0B_0C_0}(ABC)$ and $B_e' = O_{ABC}(A_0B_0C_0)$. By Sondat's theorem, $B_e', B_e, H'$ are collinear and the proposition follows since $O_0$ is the midpoint of $B_eB_e'$(a property of isogonal conjugates).
\end{proof}

We are finally ready to prove the open problem.

\begin{proof}[Proof of Theorem 1.1]
	Let $A',B',C'$ be the pedals from $B_e$ onto $AI,BI,CI$ respectively, and let $H'$ denote the orthocenter of $A'B'C'$. By proposition 4.1 and Theorem 3.1, $H'$ is the Nagel point of $ABC$ and lies on the fixed radical axis, so it remains to show that $IH'$ coincides with this radical axis.
	
	Let $O_0$ denote the circumcenter of $A_0B_0C_0$. Define $\mathcal{C}$ to be the circumconic of $A'B'C'$ that also passes through $B_e, O_0$. According to Proposition 3.4 and Corollary 3.2.2, the fixed radical axis is perpendicular to the tangent of $B_e$ or $O_0$ at $\mathcal{C}$. If $b$ is the tangent of $B_e$ at $\mathcal{C}$, then it suffices to show that $IH'\perp b$.
	
	Since the circumcenter of $A'B'C'$ is the midpoint of $IB_e$, which coincides with the circumcenter of $ABC$, therefore $OH'$ is the Euler line of $A'B'C'$.
	Suppose $OH'$ is transformed to conic $\mathcal{H}$ under isogonal conjugation w.r.t $A'B'C'$. According to Lemma 2.2 and the proof of Lemma 2.3, if $Q_0$ is the isogonal conjugate of $O_0$, then $Q_0$ lies on the Steiner's line of $B_e$ w.r.t $A'B'C'$. By Proposition 3.1, $Q_0$ lies on the Euler line of $A'B'C'$, which is $OH'$, so it follows that $O_0\in \mathcal{H}$. On the other hand, since $I, B_e$ are antipodal on $\odot(A'B'C')$, their Steiner's lines are perpendicular, so the isogonal conjugate of $I$ w.r.t $A'B'C'$ is the point of infinity along the Steiner line of $B_e$, which is $OH'$. Hence, we have $I\in \mathcal{H}$.
	
	Now suppose $b$ intersects $\odot(A'B'C')$ again at $L$(it is possible that $L\equiv B_e$ if $b$ is a tangent.) Consider the following cross-ratio equality,
	\begin{align*}
	B_e(A',B';C',L) &= B_e(A',B';C',B_e) \tag{taken on $\mathcal{C}$} \\
	&= O_0(A',B';C',B_e) \tag{taken on $\mathcal{C}$}\\
	&= O_0(A',B';C',H') \tag{by Proposition 4.2, taken on $\mathcal{H}$} \\
	&= I(A',B';C',H') \tag{taken on $\odot(A'B'C')$}.
	\end{align*}
	This implies that $I,H',L$ are collinear, so $IH'\perp b$, which is what we wanted to prove. The open problem has been resolved.
\end{proof}

\begin{figure}[t]
	\centering
	\includegraphics[scale = 0.7]{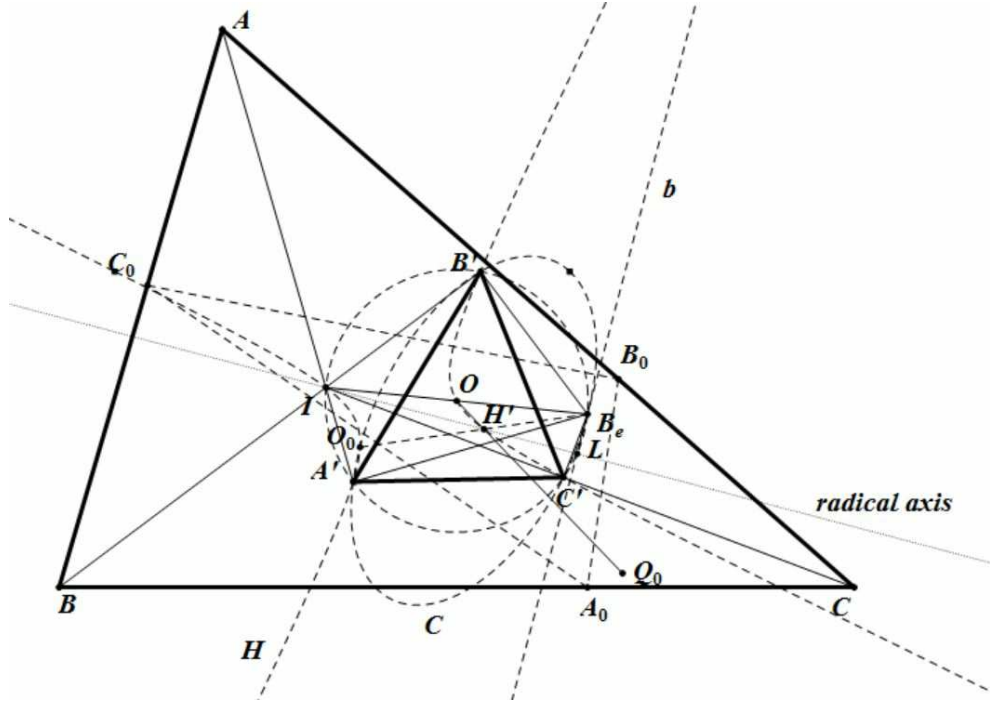}
	\caption{Proof of open problem}
	\label{T11Proof}
\end{figure}


\begin{thebibliography}{9}
	\bibitem{openproblem}
	Problem Section, \textit{Journal of Classical Geometry Volume 1}, 2012
	
	\bibitem{orthocenterlemma}
	Liang, Xuming. "New Results for Orthocenters." Dec 22, 2015.
	
	\url{https://artofproblemsolving.com/community/c2671h1176952_new_results_for_orthocentersdue_to_be_organized}
	
	\bibitem{isogonal1}
	A. V. Akopyan and A. A. Zaslavsky. \textit{Geometry of conics}, volume 26 of Mathematical World.
	American Mathematical Society, Providence, RI, 2007
		
\end{thebibliography}
\end{document}